\documentclass[reqno]{amsart}
\usepackage{color}
\def\ve{{\varepsilon}}
\def\N{{\mathbb{N}}}

\def\R{{\mathbb{R}}}

\def\BB{{\mathcal{B}}}

\newcommand{\nc}{\normalcolor}
\begin{document}
\title[Periodicity]{Infinite-horizon problems under periodicity constraint}

\author[BLOT BOUADI NAZARET]
{J.~BLOT, A.~BOUADI and B.~NAZARET}

\address{Jo\"{e}l Blot: Laboratoire SAMM EA 4543,\newline
Universit\'{e} Paris 1 Panth\'{e}on-Sorbonne, centre P.M.F.,\newline
90 rue de Tolbiac, 75634 Paris cedex 13,
France.}
\email{blot@univ-paris1.fr}
\address{Abdelkader Bouadi: D\'epartement de Technologie, \newline
Facult\'e de Technologie, \newline
Universit\'e 20 ao\^ut 1955-Skikda,\newline
BP 26 Route d'El-Hadaiek-21000,\newline
Skikda, Alg\'erie}
\email{abdelkader.bouadi14@gmail.com}
\address{Bruno Nazaret: Laboratoire SAMM EA 4543,\newline
Universit\'{e} Paris 1 Panth\'{e}on-Sorbonne, centre P.M.F.,\newline
90 rue de Tolbiac, 75634 Paris cedex 13,
France.}
\email{bruno.nazaret@univ-paris1.fr}
\date{January 29, 2016}
\begin{abstract} We study some infinite-horizon optimization problems on spaces of periodic functions, for non periodic Lagrangians. The main strategy relies on the reduction to finite horizon thanks to the introduction of an averaging operator. We then provide existence results and necessary optimality conditions, in which the corresponding averaged Lagrangian appears.
\end{abstract}

 \maketitle
\numberwithin{equation}{section}
\newtheorem{theorem}{Theorem}[section]
\newtheorem{lemma}[theorem]{Lemma}
\newtheorem{example}[theorem]{Example}
\newtheorem{remark}[theorem]{Remark}
\newtheorem{definition}[theorem]{Definition}
\newtheorem{corollary}[theorem]{Corollary}
\newtheorem{proposition}[theorem]{Proposition}
\noindent
{MSC 2010:}  49K30, 49N20.\\
{Key words:} infinite-horizon variational problem, weighted Sobolev spaces, periodic trajectory.

%%%%%%%%%%%%%%%%%%%%%%%%%%%%%%%%%%%%%%%%%%%%%%
%%%%%%%%%%%%%%%%%%%% Introduction %%%%%%%%%%%%
%%%%%%%%%%%%%%%%%%%%%%%%%%%%%%%%%%%%%%%%%%%%%%

\section{Introduction}
This paper is devoted to the study of optimal periodic trajectories in a model which is not periodic.

In the books of Colonius \cite{Co} and Kovaleva \cite{Ko} we can find a lot of problems which motivate a theory of periodic optimal control, for instance in chemical engineering, flight optimal performance, harvesting, maintenance,  or in dynamic pricing. In a different framework the question of periodicity in infinite-horizon optimal control problems is considered in \cite{AB}.
\medskip

In the variational setting of macroeconomic optimal growth theory \cite{Sa}, the model problems is usually set as the minimization of a functional taking the form
\begin{equation}\label{Eq:Model}
p\mapsto \int_0^{+ \infty} e^{-rt} L(t, p(t), p'(t)) \,dt,
\end{equation}
among functions $p:\R^+\to\R^n$ such that $p(0)=p_0\in\R^n$. Here, $r$ is a positive real number representing a discount rate (also called a rate of preference for the present by the economists). In his pioneering work \cite{Ra}, Ramsey does not use such such a discount rate since it disadvantage the future generations with respect to the present one, at the price of mathematical difficuties for existence of solutions. An alternative way to overcome this issue but still ensure an intergenerational equity in presence of a discount rate is to only permit periodic processes with a period equal to the lifespan of a generation, though the function has no reason to satisfy any periodicity condition. Then the welfare of each future generation will be the same one that the welfare of the present generation. We can extend this viewpoint to ecological models (for forests or fisheries management for instance) as considered in \cite{Cl}. Here again, ff we only permit periodic processes, we avoid overpopulation of extinction phenomenons of living species.
\medskip

The existence of a solution for problem \eqref{Eq:Model} is usually obtained under the assumtion of joint convexity in $(p,p')$ variables of the Lagrangian $L$, plus growth conditions. It would be possible to relax this assumption to convexity only on the $p'$ variable, if the measure with density $t\mapsto e^{-rt}$ satisfied a Sobolev embedding on the half-line. Unfortunately, it is known to be false (see \cite{An}). We shall overcome this difficulty by reducing the infinite-horizon problem to a finite horizon one, noticing that, for any $T$-periodic state function $p:\R^+\to\R$, we can write formally
\begin{equation}\label{Eq:Averaging}
\int_0^{+ \infty} e^{-rt} L(t, p(t), p'(t)) \,dt = \frac{1}{1-e^{-rT}}\int_0^{T} e^{-rt}{\mathcal A_1}(L)(t,p(t),p'(t))\, dt,
\end{equation}
where, for any $(t,x,y)\in[0,T]\times\R^n\times\R^n$,
\[
{\mathcal A_1}(L)(t,x,y) = {\mathcal A}\left(L(\cdot,x,y)\right)(t) = (1 - e^{-rT}) \sum_{k=0}^{+ \infty} e^{- rk T} L(t + kT,x,y).
\]
It is worth noticing that the needed Sobolev embedding holds on $[0,T]$, making us enable to use standard existence results, such as those which can be found in \cite{BGH} for instance, provided we are able to translate assumptions on ${\mathcal A}_1(L)$ into assumptions on $L$.

The operator ${\mathcal A}$ introduced above has a very interesting interpretation as a $L^2$~-~{projection} on the space of periodic functions. This motivates the study of the following simple minimization problem. Let $x:\R^+\to\R$ be a function and $T>0$ a fixed period. We want to find the (unique) solution to
\begin{equation}\label{Eq:econometrics}
\inf_{p \ T-\mbox{periodic}, \ a\in\R^n}\int_0^{+ \infty} e^{-rt}  \vert x(t) - p(t) - t a \vert^2 \,dt
\end{equation}
The problem \eqref{Eq:econometrics} addresses the problem of finding, in the sense of the least square method, the best approximation of the function $x$ as an ocsillation around a linear function. Such concepts can be in particular found in Econometrics (see \cite{Mad}), where $(p,a)$ would represent respectively the seasonality and the trend.

\medskip
There exists a litterature on the Calculus of Variations and on optimal control theory in continuous time and infinite horizon in presence of a discount rate. The unique general treatise on this theory is \cite{CHL}. For existence results, one can quote \cite{P,LPW} and references therein. The question of the necessary conditions of optimality is treated in \cite{JBPM,JBNH1} and references therein, on the subclass of the bounded trajectories in \cite{JBPCa1,JBPCa2}, while the subclass of almost-periodic trajectories appears in \cite{JBAB1}. Finally, the case of the subclass of periodic trajectories is studied in \cite{CK}. In this last paper, the authors deals criterions of the form $\int_0^{+ \infty} e^{-rt} g(x(t), u(t)) dt$, that is a $g$ which is autonomous. In our problem (\ref{Eq:Model}) $L$ depends upon $t$ in a non necessary $T$-periodic way. Concerning discrete time problems, we refer to \cite{JBNH2} and references therein.
\medskip

Now we describe the contents of the paper. In Section 2, we introduce some notations for the used function spaces, and recall some basic results. In Section 3 we deal with the orthogonal projection on a subspace of periodic functions, and solve problem \eqref{Eq:econometrics} (Theorem \ref{th35}). Section 4 is devoted to existence results on problems of the form \eqref{Eq:Model}, in Sobolev spaces of periodic functions (Theorem \ref{th25} and\ref{th44}). We end the paper by establishing some necessary conditions of optimality in problem \eqref{Eq:Model} (Theorem \ref{51}). The most important fact here is that the usual Euler-Lagrange equation is satisfied by the {\em averaged} version of the Lagrangian.

%%%%%%%%%%%%%%%%%%%%%%%%%%%%%%%%%%%%%%%%%%%%%%%%%%%%%%%%%%%%%%
%%%%%%%%%%%%%%%%% Notations and preliminaries %%%%%%%%%%%%%%%%
%%%%%%%%%%%%%%%%%%%%%%%%%%%%%%%%%%%%%%%%%%%%%%%%%%%%%%%%%%%%%%

\section{Notation and Preliminaries}

We set here some notations related to the functional framework and recall some basic facts.

Let $n\in\N^*$. For any vectors $x=(x^i)_{1\leq i\leq n}$ and $y=(y^i)_{1\leq i\leq n}$ in $\R^n$, $x\cdot y := \sum_{i=1}^nx^iy^i$ will stand for the usual Euclidean inner product and the induced norm will be denoted by $|\cdot|$.

When $X$ and $Y$ are Banach spaces, $C^0(X,Y)$ (resp. $C^1(X,Y)$) denotes the space of continuous (resp. continuously Fr\'echet-differentiable) functions from $X$ to $Y$.

The Lebesgue $\sigma$-algebra on $\R_+$ is denoted by $\overline\BB(\R_+)$. For any $r>0$, we define the measure $\mu_r$ as
\[
\forall B\in\overline\BB(\R_+), \ \mu_r(B) = \int_Be^{-rt}\,d\lambda(t),
\]
where $\lambda$ stands for the Lebesgue measure. Notice that the $\mu_r$-neglectibility of a set is, for any $r>0$, equivalent to the $\lambda$-neglectibility, thanks to the positivity of the density function.

The associated Lebesgue spaces ${\mathcal L}^\alpha(I,\mu_r,;\R^n)$ (resp. $L^\alpha(I,\mu_r;\R^n)$), with $\alpha\geq 1$ and $I$ any interval in $\R_+$, are the space of all (resp. class of) measurable $\R^n$-valued functions on $I$ whose $\alpha$th-power is $\mu_r$-integrable and the corresponding Sobolev spaces $W^{1,\alpha}(I,\mu_r;\R^n)$ are defined as
\[
W^{1,\alpha}(I,\mu_r;\R^n) := \left\{ f\in L^\alpha(I,\mu_r;\R^n); \ f'\in L^\alpha(I,\mu_r;\R^n)\right\},
\]
$f'$ being understood as the distributional first derivative of $f$. Endowed respectively with the norms
\[
\|f\|_{L^\alpha(I,\mu_r;\R^n)} := \left(\int_I|f(t)|^\alpha \,d\mu_r(t)\right)^{\frac{1}{\alpha}}
\]
and
\[
\|f\|_{W^{1,\alpha}(I,\mu_r;\R^n)} := \left(\|f\|_{L^\alpha(I,\mu_r;\R^n)}^\alpha+\|f\|_{L^\alpha(I,\mu_r;\R^n)}^\alpha\right)^{\frac{1}{\alpha}},
\]
$L^\alpha(I,\mu_r;\R^n)$ and $W^{1,\alpha}(I,\mu_r;\R^n)$ are Banach spaces, both reflexive if $s>1$. The usual Lebesgue and Sobolev spaces w.r.t. the Lebesgue measure will be simply respectively denoted by $L^\alpha(I;\R^n)$ and $W^{1,\alpha}(I;\R^n)$. Notice that if $I$ is bounded (for instance $I=[0,T]$), we have
\[
\forall r>0, \ L^\alpha(I,\mu_r;\R^n) = L^\alpha(I;\R^n) \mbox{ and } W^{1,\alpha}(I,\mu_r;\R^n) = W^{1,\alpha}(I;\R^n).
\]
Let us now introduce classical spaces of periodic functions. Let, for any $T>0$, $P_T^0(\R_+,\R^n)$ be the space of continuous $T$-periodic functions from $\R_+$ to $\R^n$ and $P^1_T(\R_+, \R^n) := P_T^0(\R_+, \R^n) \cap C^1(\R_+, \R^n)$. In addition, we define $P_{T,0}^0(\R_+, \R^n)$ (resp. $P_{T,0}^1(\R_+, \R^n)$) as the space of functions $u \in P_T^0(\R_+, \R^n)$ (resp. $P_{T,0}^1(\R_+, \R^n)$) such that $u(0)=0$.

We also recall some results on the periodic extension of a funtion $f:[0,T) \to \R^n$ defined as
\[
\forall k\in\N, \forall t\in[kT,(k+1)T), \ {\mathcal E}_T(f)(t):= f(t-kT).
\]
It is clear that
\begin{itemize}
\item[(i)] ${\mathcal E}_T$ is a linear operation.
\item[(ii)] ${\mathcal E}_T(f)$ is $T$-periodic on $\R_+$ and
\[
{\mathcal E}_T(f)\in P_T^0(\R_+,\R^n)\Longleftrightarrow\lim_{t\to T^-}f(t)=f(0).
\]
\end{itemize}
In addition, the following holds.
\begin{proposition}\label{pro25}
Let $\alpha\geq 1$ and let $\overline{P_T^0}^\alpha(\R_+,\mu_r;\R^n)$ be the closure of $P_T^0(\R_+;\R^n)$ in $L^\alpha(\R_+,\mu_r;\R^n)$. Then,
\begin{itemize}
\item[(i)] $\overline{P_T^0}^\alpha(\R_+,\mu_r;\R^n) = \left\{ f\in L^\alpha(\R_+,\mu_r;\R^n);\ f(t+T) = f(t) \mbox{ for a.e. }t\right\}$.
\item[(ii)] ${\mathcal E}_T$ is a continuous linear map from $L^\alpha([0,T),\mu_r;\R^n)$ to $\overline{P_T^0}^\alpha(\R_+,\mu_r;\R^n)$ and $\forall f\in L^\alpha([0,T),\mu_r;\R^n)$,
\begin{equation}\label{homot}
\|{\mathcal E}_T(f)\|_{L^\alpha(\R_+,\mu_r;\R^n)} = \left(\frac{1}{1- e^{-rT}}\right)^{\frac{1}{\alpha}}\|f\|_{L^\alpha([0,T[,\mu_r;\R^n)}.
\end{equation}
\end{itemize}
\end{proposition}
\begin{proof}
We know from Proposition 3 in \cite{JBAB1} that
\[
\overline{P_T^0}^\alpha(\R_+,\mu_r;\R^n) \subset \left\{ f\in L^\alpha(\R_+,\mu_r;\R^n);\ f(t+T) = f(t) \mbox{ for a.e. }t\right\}.
\]
Conversely, Let $f\in L^\alpha(\R_+,\mu_r;\R^n)$ satisfying $f(t+T) = f(t)$ for a.e. $t$. Since $C^{\infty}_c((0,T), \R^n)$ (here the space of the $C^{\infty}$ functions defined on $[0,T]$ with compact support in $(0,T)$) is dense in $L^\alpha(0,T; \R^n)$ (see \cite{Br}, Corollary 4.23, p. 109), we can find for any positive $\epsilon$ some $f{\ve} \in C^{\infty}_c((0,T), \R^n)$ such that
\[
\int_0^T | f(t) -f_{\epsilon}(t) |^\alpha \,dt\leq \varepsilon^\alpha.
\]
In addition, since the support of $f_{\varepsilon} \subset (0,T)$, we have ${\mathcal E}_T(f_{\epsilon}) \in P_T^0(\R_+, \R^n)$. We then get
\begin{equation*}
\begin{split}
& \int_0^{+\infty} e^{-rt}\vert f(t) -  {\mathcal E}_T(f_{\epsilon})(t) \vert^\alpha \,dt = \sum_{k=0}^{+\infty} \int_{kT}^{(k+1)T}e^{-rt}\vert f(t) -  {\mathcal E}_T(f_{\epsilon})(t) \vert^\alpha \,dt \\
& =  \sum_{k=0}^{+\infty} \int_0^T e^{-rt} e^{-rkT} \vert f(t+kT) -  {\mathcal E}_T(f_{\epsilon})(t+kT) \vert^\alpha \,dt\\
&= \sum_{k=0}^{+\infty} e^{-rkT} \int_0^T e^{-rt} \vert f(t) -  {f_{\epsilon}}(t) \vert^\alpha \,dt\\
& =  \frac{1}{1- e^{-rT}} \int_0^T \vert f(t) -  {f_{\epsilon}}(t) \vert^\alpha \,dt \\
& \leq \frac{1}{1- e^{-rT}} \varepsilon^\alpha.
\end{split}
\end{equation*}
That ends the proof of both (i) and (ii), since it has already been noticed that ${\mathcal E}_T$ is linear and by applyng the computation above to an arbitrary $f\in L^\alpha(\R_+,\mu_r;\R^n)$.
\end{proof}

We have an analogue result for Sobolev spaces, if we naturally restrict to $0$ Dirichlet boundary conditions, that is to the space $W^{1,\alpha}_0([0,T);\R^n)$ of functions $f$ in $W^{1,\alpha}([0,T);\R^n)$ such that $f(0)=f(T)=0$.

\begin{proposition}\label{pro26} For any $\alpha\geq 1$, let $W^{1,\alpha}_{T,0}(\R_+, \mu_r; \R^n)$ be the closure of the space $P_{T,0}^1(\R_+, \R^n)$ in $W^{1,\alpha}(\R_+, \mu_r; \R^n)$. Then,
\[
W^{1,\alpha}_{T,0}(\R_+, \mu_r; \R^n) = W^{1,\alpha}(\R_+, \mu-r; \R^n) \cap P^0_{T,0}(\R_+, \R^n).
\]
Moreover, ${\mathcal E}_T$ is a continuous linear map from $W^{1,\alpha}_0([0,T);\R^n)$ to $W^{1,\alpha}_{T,0}(\R_+, \mu_r; \R^n)$ and, for any $f\in W^{1,\alpha}_0([0,T);\R^n)$,
\begin{equation}\label{sob-extend}
\|{\mathcal E}_T(f)\|_{W^{1,\alpha}(\R_+, \mu_r; \R^n)} \leq \left(\frac{1}{1- e^{-rT}}\right)^{\frac{1}{\alpha}}\|f\|_{W^{1,\alpha}([0,T);\R^n)}.
\end{equation}
\end{proposition}

\begin{proof} 
Let us first take $f\in W^{1,\alpha}_{T,0}(\R_+, \mu_r; \R^n)$. By definition, $f\in W^{1,\alpha}(\R_+, \mu_r; \R^n)$ and we can find a sequence $(f_m)_{m\in\N}$ with values in $P_{T,0}^1(\R_+, \R^n)$ converging in $W^{1,\alpha}(\R_+, \mu_r; \R^n)$ to $f$. Since, for any $K>0$ and any $B\in \overline\BB([0,T])$, we have
\[
e^{-rK}\mu_0(B)\leq \mu_r(B)\leq \mu_0(B),
\]
and, using standard Sobolev embeddings on the real line, we get the convergence in $C^0([0,K],\R^n)$, for any $K>0$. The conclusion $f\in P^0_{T,0}(\R_+, \R^n)$ follows immediately.
Conversely, let $f\in W^{1,\alpha}(\R_+, \mu_r; \R^n) \cap P^0_{T,0}(\R_+, \R^n)$ and set
\[
f_T = f_{|(0,T)} \in W^{1,\alpha}_0((0,T),\R^n).
\]
Then, for any $\ve>0$, there exists $\varphi_\ve\in C^1_c((0,T),\R^n)$ such that
\[\| f_T-\varphi_\ve\|_{W^{1,\alpha}_0((0,T),\R^n)}\leq \ve.
\]
Extending $\varphi_\ve$ on $[0,T]$ by setting $\varphi_\ve(0)=\varphi_\ve(T)=0$ yields ${\mathcal E}_T(\varphi_\ve)\in P_{T,0}^1(\R_+, \R^n)$. In addition, since $f\in P^0_{T,0}(\R_+, \R^n)$, we have ${\mathcal E}_T(f_T)=f$. Using Proposition~\ref{pro25} and the fact that ${\mathcal E}_T(\varphi_\ve)'={\mathcal E}_T(\varphi_\ve')$ a.e. on $\R_+$, we finally get
\[
\|f-{\mathcal E}_T(\varphi_\ve)\|_{W^{1,\alpha}(\R_+,\mu_r,\R^n)} \leq \left(\frac{1}{1- e^{-rT}}\right)^{\frac{1}{\alpha}}\ve,
\]
and the conclusion. The inequality \eqref{sob-extend} is then an easy consequence of Proposition \ref{pro25}.
\end{proof}

\vskip5mm
%%%%%%%%%%%%%%%%%%%%%%%%%%%%%%%%%%%%%%%%%%%%%%%%%%%%%%%%%%%%%%%%%%%%%%%%%
\section{The $L^2$-Projection on a Lebesgue space of periodic functions}%
%%%%%%%%%%%%%%%%%%%%%%%%%%%%%%%%%%%%%%%%%%%%%%%%%%%%%%%%%%%%%%%%%%%%%%%%%

%%%%%%%%%%%%%%%%%%%%%%%%%%%%%%%%%%%%
\subsection{The averaging operator}%
%%%%%%%%%%%%%%%%%%%%%%%%%%%%%%%%%%%%

We introduce and study here the main tool we will use in order to reduce variational problems that are set on $\R_+$ to finite horizon.

\begin{theorem}\label{def-averaging}
When $g : \R_+ \rightarrow \R^n$ and when $s \in [0,T]$, we set 
\[
\mathcal A(g) := \left[ s \mapsto (1 - e^{-rT}) \sum_{k=0}^{+ \infty} e^{- rk T} g(s + kT)\right]
\] 
when it is defined. Let $\alpha \in [1, + \infty)$. Then the following assertion holds:
\[
g \in {\mathcal L}^{\alpha}(\R_+, \mu_r; \R^n) = {\mathcal L}^{\alpha}(\R_+, \overline{\mathfrak B}(\R_+), \mu_r; \R^n) \Longrightarrow A(g) \in 
{\mathcal L}^{\alpha}(0,T; \R^n)
\]
 and moreover, for all  $g \in {\mathcal L}^{\alpha}(\R_+, \mu_r; \R^n)$, we have 
$$\Vert A(g)
\Vert_{L^{\alpha}(0,T)} \leq \left( \frac{1- e^{-rT}}{e^{-rT}} \right)^{\frac{1}{\alpha}} \Vert g \Vert_{L^{\alpha}(\R_+, \mu_r)}.$$
And so $A$ is a linear bounded operator from  $L^{\alpha}(\R_+, \mu_r; \R^n)$ into $L^{\alpha}(0,T; \R^n)$.
\end{theorem}

\noindent
\begin{proof} We treat separately the cases $\alpha=1$ and $\alpha>1$.

\medskip
{\bf The Case $\alpha$ = 1.} Let $g \in  {\mathcal L}^{1}(\R_+, \mu_r; \R^n)$. Using the $\sigma$-additivity of the positive measure of density $[t \mapsto e^{-r t}  |g(t)|]$ with respect to the Lebesgue measure $\lambda$ and $\lambda(\{(k+1)T)\})=0$ for any integer $k$, we obtain 
\begin{equation*}
\begin{split}
\sum_{k = 0}^{+ \infty} \int_{[kT, kT + T)}e^{-rt} |g(t)|d\lambda(t) & = \int_{\bigcup_{ k \in \N} [kT, kT + T)}  e^{-rt} |g(t)| \lambda(t)\\
& = \int_{\R_+} e^{-rt} |g(t)| \lambda(t) < +\infty
\end{split}
\end{equation*}
hence

$$\sum_{k = 0}^{+ \infty} \int_{[kT, kT + T]}e^{-rt} |g(t)| d \lambda(t) < + \infty.$$

Doing a change of variable on each term of this sum, we obtain
\begin{equation}\label{eq1}
\sum_{k = 0}^{+ \infty} \int_{[0, T]}e^{-rs} e^{- rkT} |g(s + kT)|d \lambda (s)  = \Vert g  \Vert_{L^{1}(\R_+, \mu_r)} < + \infty,
\end{equation}
and using the linearity of the integral we obtain
\begin{equation}\label{eq2}
\lim_{m \rightarrow + \infty} \int_{[0,T]} e^{-rs} \sum_{k=0}^{+ \infty}  e^{- rkT} | g(s + kT) | d \lambda (s) < + \infty.
\end{equation}
Setting $\phi_m(s) := e^{-rs} \sum_{k=0}^{m}  e^{- rkT} |g(s + kT)|$, we have $\phi_m \in {\mathcal L}^1(0,T; \R_+)$. The sequence $(\phi_m)_{m \in \N}$ is non decreasing, and $\sup_{ m \in \N} \int_{[0,T]} \phi_m(s) d \lambda (s) < + \infty$. Then we can use the B. Levi theorem and assert that 
$\lim_{m \rightarrow + \infty} \phi_m \in {\mathcal L}^1 (0,T; \R_+)$, i.e. $[ s \mapsto  e^{-rs} \sum_{k=0}^{+ \infty}  e^{- rkT} |g(s + kT)| ]  \in {\mathcal L}^1 (0,T; \R_+)$ which implies that this function is $\lambda$-a.e. finite (i.e. the series is convergent in $\R$ for $\lambda$-a.e. $s \in [0,T]$). Since the absolute convergence of series implies the convergence, we can write, for  $\lambda$-a.e. $s \in [0,T]$,
\begin{equation}\label{eq3}
\left|  e^{-rs} \sum_{k=0}^{+ \infty}  e^{- rkT} g(s + kT) \right| \leq  e^{-rs} \sum_{k=0}^{+ \infty}  e^{- rkT} |g(s + kT)|,
\end{equation}
and since $[s \mapsto e^{-rs} \sum_{k=0}^{+ \infty}  e^{- rkT} g(s + kT) ]$ is Lebesgue measurable, it also belongs to $ {\mathcal L}^1 (0,T; \R_+)$. Therefore we obtain that ${\mathcal A}(g) \in {\mathcal L}^1 (0,T; \R_+)$. Moreover from (\ref{eq3}) we obtain 
and since $[s \mapsto e^{-rs} \sum_{k=0}^{+ \infty}  e^{- rkT} g(s + kT) ]$ is Lebesgue measurable, it also belongs to $ {\mathcal L}^1 (0,T; \R_+)$. Therefore we obtain that ${\mathcal A}(g) \in {\mathcal L}^1 (0,T; \R_+)$. Moreover from (\ref{eq3}) we obtain 
\begin{multline*}
 \int_{[0,T]}  e^{-rT} \left|\sum_{k=0}^{+ \infty}  e^{- rkT} g(s + kT)\right| d \lambda (s)   \leq   \int_{[0,T]}  e^{-rs} \left|\sum_{k=0}^{+ \infty}  e^{- rkT} g(s + kT)\right| d \lambda (s) \\
\leq \int_{[0,T]} e^{-rs} \sum_{k=0}^{+ \infty}  e^{- rkT} |g(s + kT)| d \lambda (s)
\end{multline*}
and using (\ref{eq1}), 
\begin{equation}\label{eq4}
\Vert {\mathcal A}(g) \Vert_{L^1(0,T)} \leq \frac{1 - e^{-rT}}{e^{-rT}}  \Vert g  \Vert_{L^{1}(\R_+, \mu_r)}.
\end{equation}

\medskip
{\bf The Case $\alpha > 1$.} Let $g \in  {\mathcal L}^{\alpha}(\R_+, \mu_r; \R^n)$. Since $\mu_r(\R_+) < + \infty$, we have ${\mathcal L}^{\alpha}(\R_+, \mu_r; \R^n) \subset {\mathcal L}^{1}(\R_+, \mu_r; \R^n)$, and so $\sum_{k=0}^{+ \infty}  e^{- rkT}  g(s + kT)$ is well defined in $\R^n$ for $\lambda$-a.e. $s \in [0,T]$.\\
 Since $| g |^{\alpha} \in {\mathcal L}^{1}(\R_+, \mu_r; \R^n)$, we can assert from the case treated above that
\[
\sum_{k=0}^{+ \infty}  e^{- rkT}  |g(s + kT)|^{\alpha}
\]
is well defined in $\R$ for $\lambda$-a.e. $s \in [0,T]$. Now, using the H\"older inequality, with $\beta > 0$ such that $\frac{1}{\alpha} + \frac{1}{\beta} = 1$, we obtain, for $\lambda$-a.e. $s \in [0,T]$, 
\[
\left| \sum_{k=0}^{+ \infty}  e^{- rkT}  g(s + kT) \right| \leq \left( \sum_{k=0}^{+ \infty}  e^{- rkT}  | g(s + kT) |^{\alpha} \right)^{\frac{1}{\alpha}}  \left(\sum_{k=0}^{+ \infty} e^{- rkT} \right)^{\frac{1}{\beta}}
\]
which leads to
\begin{equation}\label{eq5}
\left| \sum_{k=0}^{+ \infty}  e^{- rkT}  g(s + kT) \right|^{\alpha} \leq \left(\frac{1}{1 - e^{-rT}} \right)^{\frac{1}{\alpha -1}} \sum_{k=0}^{+ \infty}  e^{- rkT}  \left| g(s + kT) \right|^{\alpha}.
\end{equation}
Applying the case $\alpha=1$ to $|g|^{\alpha}$ we obtain that $[s \mapsto \sum_{k=0}^{+ \infty}  e^{- rkT}  | g(s + kT)|^{\alpha}] \in {\mathcal L}^1 (0,T; \R_+)$, and consequently from (\ref{eq5}) we obtain that
\[
[s \mapsto \sum_{k=0}^{+ \infty}  e^{- rkT}  g(s + kT)] \in {\mathcal L}^{\alpha} (0,T; \R^n)
\]
which implies that ${\mathcal A}(g) \in {\mathcal L}^{\alpha} (0,T; \R^n)$. Now we prove the announced inequality. First notice that, using the Beppo Levi theorem, we have
$$\sum_{k = 0}^{+ \infty} \int_{[0, T]}e^{-rs} e^{- rkT} | g(s + kT) | d \lambda (s) = \int_{[0, T]}e^{-rs} \left( \sum_{k = 0}^{+ \infty}e^{- rkT} | g(s + kT)| \right)  d \lambda (s),$$
and using (\ref{eq5}) and (\ref{eq1}) we obtain
\[
\begin{split}
& e^{-rT} \!\int_{[0,T]} \frac{1}{(1- e^{-rT} )^{\alpha}} \left| {\mathcal A}(g)(s) \right|^{\alpha} d \lambda(s) = \int_{[0,T]} \!e^{-rT} \left| \sum_{k = 0}^{+ \infty}e^{- rkT} g(s+ kT) \right|^{\alpha} d \lambda (s)\\
\leq & \int_{[0,T]} e^{-rs} \left| \sum_{k = 0}^{+ \infty}e^{- rkT} g(s+ kT) \right|^{\alpha} d \lambda (s)\\
\leq & \left(\frac{1}{1 - e^{-rT}} \right)^{\frac{1}{\alpha -1}} \int_{[0,T]} e^{-rs} \left( \sum_{k=0}^{+ \infty}  e^{- rkT} |g(s + kT)|^{\alpha} \right) d \lambda(s) \\
= & \left(\frac{1}{1 - e^{-rT}} \right)^{\frac{1}{\alpha -1}} \sum_{k=0}^{+ \infty} \int_{[0,T]} e^{-rs}  e^{- rkT} |  g(s + kT) |^{\alpha} d \lambda(s) \\
= &  \left(\frac{1}{1 - e^{-rT}} \right)^{\frac{1}{\alpha -1}} \Vert \; | g |^{\alpha} \; \Vert_{L^1(\R_+, \mu_r)} \\
= &  \left(\frac{1}{1 - e^{-rT}} \right)^{\frac{1}{\alpha -1}} \Vert g \Vert^{\alpha}_{L^{\alpha}(\R_+, \mu_r)}
\end{split}
\]
\vskip2mm
\noindent
which implies
$$\frac{ e^{-rT}}{(1-  e^{-rT})^{\alpha}} \int_{[0,T]} \left| A(g)(s) \right|^{\alpha} d \lambda(s) \leq \frac{1}{(1-  e^{-rT})^{\alpha - 1}} \Vert g \Vert^{\alpha}_{L^{\alpha}(\R_+, \mu_r)}$$
The conclusion immediately follows.
\end{proof}

%%%%%%%%%%%%%%%%%%%%%%%%%%%%%%%%%
\subsection{Projection in $L^2$}%
%%%%%%%%%%%%%%%%%%%%%%%%%%%%%%%%%

In this section, we first provide an explicit formula of the orthogonal projection of a function $f$ in $L^2(\R_+, \mu_r; \R^n)$ on the Lebesgue space of $T$-periodic functions $\overline{P_T^0}(\R_+, \mu_r; \R^n)$. Secondly we give a rigorous formulation of problem (\ref{Eq:econometrics}) and establish an existence result, giving an explicit formula for the solution.
\begin{theorem}\label{th31}
Let $f \in L^2(\R_+, \mu_r; \R^n)$. The orthogonal projection of $f$ on the subspace $\overline{P_T^0}(\R_+, \mu_r; \R^n)$ is $\left({\mathcal E}_T\circ{\mathcal A}\right)f$.
\end{theorem}
\begin{proof} Using Proposition \ref{pro25}, we see that $\left({\mathcal E}_T\circ{\mathcal A}\right)f \in \overline{P_T^0}(\R_+,\mu_r; \R^n)$, which is a closed vector subspace of the Hilbertian space $L^2(\R_+, \mu_r; \R^n)$. It follows that the orthogonal projection of $f$ on $\overline{P_T^0}(\R_+,\mu_r; \R^n)$ exists and is unique. Let us denote by $p$ this orthogonal projection. It is characterized by the following property
$$\forall q \in \overline{P_T^0}(\R_+,\mu_r; \R^n), (f-p \mid q )_{L^2(\R_+,\mu_r; \R^n)} = 0,$$
that is
\begin{equation}\label{eq31}
\forall q \in \overline{P_T^0}(\R_+,\mu_r; \R^n), (f \mid q)_{L^2(\R_+,\mu_r; \R^n)} = (p \mid q)_{L^2(\R_+,\mu_r; \R^n)}.
\end{equation}
Let $q \in \overline{P_T^0}(\R_+,\mu_r; \R^n)$, then arguing as in Proposition \ref{pro26} and using the periodicity of $q$ we obtain 
$$(f \mid q)_{L^2(\R_+,\mu_r; \R^n)} = \int_0^T e^{-rs}( \sum_{k=0}^{+\infty}e^{-rkT} f(s+kT))\cdot q(s) ds.$$
This last equality implies that, for any $q \in \overline{P_T^0}(\R_+,\mu_r; \R^n)$,
\begin{equation}\label{eq32}
(f \mid q)_{L^2(\R_+,\mu_r; \R^n)} = \frac{1}{1-e^{-rT}} \int_0^T e^{-rs} ({\mathcal A}f)(s)\cdot q(s) ds.
\end{equation}
Replacing $f$ by an arbitrary $g \in \overline{P_T^0}(\R_+,\mu_r; \R^n)$ in the previous computation, we obtain
$$(g \mid q)_{L^2(\R_+,\mu_r; \R^n)}  =   \int_0^T e^{-rs} \frac{1}{1-e^{-rT}} g(s)\cdot q(s) ds,$$
that implies, for any $g$ and $q\in\overline{P_T^0}(\R_+,\mu_r; \R^n)$,
\begin{equation}\label{eq33}
(g \mid q)_{L^2(\R_+,\mu_r; \R^n)} = \frac{1}{1-e^{-rT}} \int_0^T e^{-rs} g(s)\cdot q(s) ds.
\end{equation}
Taking $g = \left({\mathcal E}_T\circ{\mathcal A}\right)f$ in (\ref{eq33}) leads, for any  to $q \in \overline{P_T^0}(\R_+,\mu_r; \R^n)$,
$$(\left({\mathcal E}_T\circ{\mathcal A}\right)f \mid q)_{L^2(\R_+,\mu_r; \R^n)} = \frac{1}{1-e^{-rT}} \int_0^T e^{-rs} M_f(s)\cdot q(s) ds.$$
Finally, we get from \eqref{eq32} that, for all  $q \in \overline{P_T^0}(\R_+,\mu_r; \R^n)$,
\[
(f \mid q)_{L^2(\R_+,\mu_r; \R^n)} = (\left({\mathcal E}_T\circ{\mathcal A}\right)f \mid q)_{L^2(\R_+,\mu_r; \R^n)},
\]
and conclude using (\ref{eq31}).
\end{proof}
\begin{corollary}\label{cor32}
When $ f \in L^2(\R_+, \mu_r; \R^n)$ the two following assertions are equivalent.
\begin{enumerate}
\item[(i)] $f$ is orthogonal to  $\overline{P_T^0}(\R_+,\mu_r; \R^n)$ in $L^2(\R_+, \mu_r; \R^n)$.
\item[(ii)] For a.e. $s \in [0,T)$, $\sum_{k=0}^{+\infty} e^{-rkT} f(s + kT) = 0$.
\end{enumerate}
\end{corollary}
\vskip4mm

\begin{proof}
 (i) is equivalent to $\left({\mathcal E}_T\circ{\mathcal A}\right)f = 0$. By definition, it is equivalent to ${\mathcal A}f = 0$ on $[0,T)$ hence, by Theorem \ref{th31}, to assertion (ii).
\end{proof}
\begin{remark}\label{rem33}
If $f$ is orthogonal to $\overline{P_T^0}(\R_+,\mu_r; \R^n)$ in $L^2(\R_+, \mu_r; \R^n)$ and if $f \geq 0$ on $\R_+$, then using Corollary \ref{cor32} we have $f(s+kT) = 0$ for all $k \in \N$ and $\lambda$-a.e. $s \in [0,T)$ that implies that $f=0$ $\mu$-a.e. on $\R_+$.
\end{remark}
\vskip4mm
Now we denote by ${\mathfrak L}(\R_+, \R^n)$ the space of functions from $\R_+$ into $\R^n$ of the form $\underline{a} := [t \mapsto t a]$ where $a \in \R^n$. ${\mathfrak L}(\R_+, \R^n)$ is a vector subspace of $L^2(\R_+, \mu_r; \R^n)$ which is isomorphic to $\R^n$, and so it has finite dimension.
\vskip3mm
\begin{lemma}\label{lem34}
The direct sum $\overline{P_T^0}(\R_+, \mu_r; \R^n) \oplus {\mathfrak L}(\R_+, \R^n)$ is a closed vector subspace of $L^2(\R_+, \mu_r; \R^n)$.
\end{lemma}
\vskip3mm
\begin{proof}
It is an immediate consequence of Corollaire, p. 229ç in \cite{TGAF}\nc.
\end{proof} 
\vskip4mm
Let us now give a rigorous formulation of problem \eqref{Eq:econometrics}. For any fixed function $x \in L^2(\R_+, \mu_r; \R^n)$, we consider the following minimization problemù.
\vskip2mm
\begin{equation}\label{eq34}
\left.
\begin{array}{rl}
\displaystyle{\rm Minimize}& E(p, \underline{a}) := \int_0^{+ \infty} e^{-rt} \vert x(t) - p(t) - t a \vert^2 dt\\
{\rm when} & p \in \overline{P_T^0}(\R_+, \mu_r; \R^n), a\in\R^n.
\end{array}
\right\}
\end{equation}
\vskip4mm
\begin{theorem}\label{th35}
For any $x \in L^2(\R_+, \mu_r; \R^n)$, the problem \eqref{eq34} admits an unique solution
\[
(\hat{p}, a) \in \overline{P_T^0}(\R_+, \mu_r; \R^n) \times \R^n
\] 
In addition, it is given by
\begin{eqnarray}
\label{Eq:a_opt}\hat{a} & = & \frac{r}{T} \int_0^T e^{-rs}(\tilde{\mathcal A}x(s) - {\mathcal A}x(s))ds,\\
\label{Eq:p_opt}\hat{p} & = & \left({\mathcal E}_T\circ{\mathcal A}\right)x - \left({\mathcal E}_T\circ{\mathcal A}\right)\hat{\underline{a}},
\end{eqnarray}
where
$$\tilde{\mathcal A}x(s) := \frac{(1-e^{-rT})^2}{ e^{-rT}} \sum_{k=0}^{+ \infty}ke^{-rkT}  x(s + kT)$$
and
\[
\underline{\hat a}(t)=\hat  at.
\]
\end{theorem}
\vskip5mm
\begin{proof}
Since, by Lemma \ref{lem34}, $\overline{P_T^0}(\R_+, \mu_r; \R^n) \oplus {\mathfrak L}(\R_+, \R^n)$ is a closed vector subspace of the Hilbert $L^2(\R_+, \mu_r; \R^n)$, the existence and the uniqueness of a solution $(\hat p,\hat a)$ of problem (\ref{eq34}) is simply due to the theorem of the orthogonal projection on a closed vector subspace in a Hilbert space. 

Let us first write
$$
\begin{array}{rcl}
E(\hat{p}, \hat{{a}}) &=& \inf \left\{ E(p, a) : p \in \overline{P_T^0}(\R_+, \mu_r; \R^n), a \in \R^n\right\}\\
 &=& \inf_{a \in \R^n}\left( \inf_{p \in \overline{P_T^0}(\R_+, \mu_r; \R^n)}E(p, a)\right).
\end{array}
$$
If $p_{\underline{a}}$ denotes the orthogonal projection of $x-\underline{a}$ on $\overline{P_T^0}(\R_+, \mu_r; \R^n)$, then, for all $a \in \R^n$, we have $\inf_{p \in \overline{P_T^0}(\R_+, \mu_r; \R^n)}E(p,a) = E(p_{\underline{a}}, a)$, and consequently
\begin{equation}\label{eq35}
E(\hat{p}, \hat{\underline{a}}) = \inf_{\underline{a} \in {\mathcal L}(\R_+, \R^n)}E(p_{\underline{a}}, \underline{a}).
\end{equation}
Since $E(p_{\hat{\underline{a}}}, \hat{{a}}) = \inf_{p \in \overline{P_T^0}(\R_+, \mu_r; \R^n)}E(p, \hat{{a}}) \leq E(\hat{p}, \hat{{a}})$, and since $(p_{\hat{\underline{a}}}, \hat{{a}})$ is optimal we have $E(\hat{p}, \hat{{a}}) \leq E(p_{\hat{\underline{a}}}, \hat{{a}})$ that implies $E(p_{\hat{\underline{a}}}, \hat{{a}}) = E(\hat{p}, \hat{{a}})$. Using the uniqueness of the optimal solution we get
\begin{equation}\label{eq36}
\hat{p} = p_{\hat{\underline{a}}}.
\end{equation}
Theorem~\ref{th31} applied to $(x-{\hat{\underline{a}}})$ then leads to 
\[
p_{\hat{\underline{a}}} = \left({\mathcal E}_T\circ{\mathcal A}\right)(x-{\hat{\underline{a}}}) = \left({\mathcal E}_T\circ{\mathcal A}\right)x - \left({\mathcal E}_T\circ{\mathcal A}\right){\hat{\underline{a}}}.
\]
This proves \eqref{Eq:p_opt}. Next, let us write, for any $a\in\R^n$, for any $s\in[0,T)$
\begin{eqnarray}\label{Eq:projx-a}
p_{{\underline{a}}}(s) &= & {\mathcal A}x(s) - (1-e^{-rT})\sum_{k=0}^{+ \infty} e^{-rkT} (s + kt) {a}\nonumber\\
 &=& {\mathcal A}x(s) - (1-e^{-rT})\sum_{k=0}^{+ \infty} e^{-rkT} s{a} - (1-e^{-rT})\sum_{k=0}^{+ \infty} e^{-rkT}(kT) {a}\nonumber\\
 &=& {\mathcal A}x(s) - s{a} -(1-e^{-rT}) \frac{T e^{-rT}}{(1-e^{-rT})^2} {a},\nonumber\\
 &=& {\mathcal A}x(s) - as - \frac{aT e^{-rT}}{(1-e^{-rT})}.
\end{eqnarray}
It remains to prove the formula for $\hat{a}$. Let us introduce the function $F : \R^n \rightarrow \R$ defined by $F(a) := E(p_{\underline{a}}, a)$. Using (\ref{eq35}) and (\ref{eq36}), we see that
\[
F(\hat{a}) = \inf_{a \in \R^n}F(a).
\]
We use here again that $p_{{\underline{a}}}=\left({\mathcal E}_T\circ{\mathcal A}\right)x - \left({\mathcal E}_T\circ{\mathcal A}\right){{\underline{a}}}$ and \eqref{Eq:projx-a} to write
\begin{multline}\label{eq38}
F(a) = \int_0^{+\infty}e^{-rt} \vert x(t) - \left({\mathcal E}_T\circ{\mathcal A}\right)x(s) + \left({\mathcal E}_T\circ{\mathcal A}\right){\underline{a}}(t) - ta \vert^2 dt\\
 = \sum_{k=0}^{+ \infty}\int_0^T e^{-rs} e^{-rkT} \vert x(s+kT)- {\mathcal A}x(s) + \frac{T e^{-rT}}{(1-e^{-rT})} a - kT a \vert^2 ds.
\end{multline}
The function $F$ is quadratic so its minimizer $\hat{a}$ can be characterized as a critical point, that is
\[
\sum_{k=0}^{+ \infty}\int_0^T e^{-rs} e^{-rkT}( \frac{T e^{-rT}}{(1-e^{-rT})} - kT) ( x(s+kT) - {\mathcal A}x(s) + \frac{T e^{-rT}}{(1-e^{-rT})} \hat{a} - kT \hat{a})ds = 0,
\]
which is equivalent to
\begin{multline*}
\int_0^T e^{-rs} \sum_{k=0}^{+ \infty}e^{-rkT}( \frac{T e^{-rT}}{(1-e^{-rT})} - kT)({\mathcal A}x(s) - x(s+kT))ds\\
= \left(\int_0^T e^{-rs} \sum_{k=0}^{+ \infty}e^{-rkT}( \frac{T e^{-rT}}{(1-e^{-rT})} - kT)^2 ds\right)\hat{a}.
\end{multline*}
A quite lengthy but straightforward computation using standard series finally leads to the expression \eqref{Eq:a_opt}.
\end{proof}
\section{Existence results for Problem \eqref{Eq:Model}}
We start by establishing some properties on the operator ${\mathcal A}_1$, defined for any function $L:\R_+\times\R^n\times\R^n$ by
\[
\forall (s,x,y)\in[0,T]\times\R^n\times\R^n, \ {\mathcal A}_1(L)(s,x,y)={\mathcal A}\left(L(\cdot,x,y)\right)(s).
\]
\begin{lemma}\label{lem41} Let $L : \R_+ \times \R^n \times \R^n \rightarrow \R_+$.
\begin{enumerate}
\item[($\alpha$)] If, for a.e. $t \in \R_+$, $L(t,.,.)$ is lower semi-continuous on $\R^n \times \R^n$, then, for a.e. $s \in [0,T]$, ${\mathcal A}_1(L)(s,.,.)$ is 
 lower semi-continuous on $\R^n \times \R^n$.
\item[($\beta$)] If, for a.e. $t \in \R_+$, $L(t,.,.)$ is upper semi-continuous on $\R^n \times \R^n$, then, for a.e. $s \in [0,T]$, ${\mathcal A}_1(L)(s,.,.)$ is 
 upper semi-continuous on $\R^n \times \R^n$.
\item[($\gamma$)] If, for a.e. $t \in \R_+$, $L(t,.,.)$ is continuous on $\R^n \times \R^n$, then, for a.e. $s \in [0,T]$, ${\mathcal A}_1(L)(s,.,.)$ is 
continuous on $\R^n \times \R^n$.
\item[($\delta$)] If,  for a.e. $t \in \R_+$ and for all $x \in \R^n$, the function $L(t,x,.)$ is convex, then, for a.e. $s \in [0,T]$, and for all $x \in \R^n$, the function ${\mathcal A}_1(L)(s,x,.)$ is convex.
\item[($\epsilon$)] let $\rho : \R^n \rightarrow \R_+$ such that, for all $(t,x,y) \in \R_+ \times \R^n \times \R^n$, $L(t,x,y) \geq \rho(y)$, then we have, for all $(s,x,y) \in [0,T] \times \R^n \times \R^n$, ${\mathcal A}_1(L)(s,x,y) \geq \rho(y)$.
\item[($\zeta$)] If $L$ is measurable from $(\R_+ \times \R^n \times \R^n, \overline{{\mathfrak B}}(\R_+) \otimes  \overline{{\mathfrak B}}(\R^n) \otimes   \overline{{\mathfrak B}}(\R^n) )$ into $(\R_+,  \overline{{\mathfrak B}}(\R_+)$, then ${\mathcal A}_1(L)$ is measurable from $([0,T] \times \R^n \times \R^n, \overline{{\mathfrak B}}([0,T]) \otimes  \overline{{\mathfrak B}}(\R^n) \otimes   \overline{{\mathfrak B}}(\R^n) )$ into $([0, + \infty],  \overline{{\mathfrak B}}([0,T]))$.
\item[($\eta$)] If $L$ is a Caratheodory function then ${\mathcal A}_1(L)$ is a Caratheodory function.
\end{enumerate}
\end{lemma}
\begin{proof} 

($\alpha$) We first establish the following assertion.
\begin{equation}\label{eq16}
\forall A \in {\mathfrak B}(\R_+), \forall p \in \R_+ \; {\rm s.t.} \; A - p \subset \R_+, \; \mu_r(A-p) \leq e^{-rp} \mu_r(A).
\end{equation}
Indeed, we have 
\begin{eqnarray*}
\mu_r(A-p) &=& \int_{A-p} 1 d \mu_r(t) = \int_{\R_+} 1_{A-p}(t) d \mu_r (t)\\
\null & = & \int_{\R_+} 1_{A}(t+p) d \mu_r(t) = \int_0^{+ \infty} e^{-rt}  1_{A}(t+p) d t\\
\null & = & \int_p^{+ \infty} e^{-rs} e^{-rp} 1_A(s) ds\\
\null & \leq & e^{-rp}\int_0^{+ \infty} e^{-rs} 1_A(s) ds =  e^{-rp}\mu_r(A),
\end{eqnarray*}
and \eqref{eq16} follows. Next, we prove that
\begin{equation}\label{eq17}
\forall B \in {\mathfrak B}([0,T]), \; \mu_r(B) = 0 \Longrightarrow \mu(B) = 0.
\end{equation}
We have 
\begin{eqnarray*}
0 &= & \mu_r(B) = \int_{\R_+} 1_B(t) d \mu_r(t) = \int_{[0,T]} 1_B(t) d \mu_r(t)\\
\null & = & \int-0^T e^{-rt} 1_B(t) dt \geq e^{-rT}  \int-0^T  1_B(t) dt\\
\null & = & e^{-rT} \mu(B) \geq 0 \Longrightarrow \mu(B) = 0.
\end{eqnarray*}
So (\ref{eq17}) is proven.  Let us now set
$$S := \{ t \in \R_+ : L(t, .,.) \; {\rm is} \; {\rm l.s.c.} \},$$
and
$$S_1 := \{ s \in [0,T] : \forall k \in \N, s + kT \in S \}.$$
From the assumption we know that $\R_+ \setminus S$ is $\mu_r$-negligible. Notice that 
\begin{eqnarray*}
[0,T] \setminus S_1 &=& \{ s \in [0,T]; \exists k \in \N \; {\rm s.t.} \; s \in \R_+ \setminus S - kT \}\\
\null & = & \bigcup_{k \in \N} ((\R_+ \setminus S) - kT).
\end{eqnarray*}
Since $(\R_+ \setminus S)$ is $\mu_r$-negligible, there exists $Z \in {\mathfrak B}(\R_+)$ such that  $(\R_+ \setminus S) \subset Z$ and $\mu_r(Z) = 0$. From \eqref{eq16}, we obtain that $\mu_r(Z - kT) = 0$ which implies that $((\R_+ \setminus S) - kT)$ is $\mu_r$-negligible for all $k \in \N$. Since a countable union of $\mu_r$-negligible sets is $\mu_r$-negligible, we obtain that $[0,T] \setminus S_1$ is $\mu_r$-negligible. Therefore there exists $W \in {\mathfrak B}(\R_+)$ such that $[0,T] \setminus S_1 \subset W$ and $\mu_r(W)=0$. Replacing $W$ by $W \cap [0,T]$ we can assume that $W \subset [0,T]$. This leads with \eqref{eq17} to $\mu(B) = 0$, and consequently we can say that $[0,T] \setminus S_1$ is $\mu$-negligible. Consequently, for a.e. $s \in [0,T]$, $L(s+ kt,.,.)$ is l.s.c. for all $k \in \N$.
Let $\chi_0 : 2^{\N} \rightarrow [0, + \infty]$ be the counting measure on $\N$ and let $\chi$ be the positive measure with density $\xi$ with respect to $\chi_0$, where $\xi = [k \mapsto \xi_k]$, from $\N$ into $\R_+$, is defined by $\xi_k :=  (1- e^{-rT}) e^{-rkT}$. We then have 
$${\mathcal A}_1(L)(s,x,y) = \int_{\N} L(s + kT,x,y) d \chi(k).$$
Let us arbitrarily fix $(s,x,y) \in [0,T] \times \R^n \times \R^n$. Let $(x_q,y_q)$ be a sequence into $\R^n \times \R^n$ which converges to $(x,y)$. For all $q \in \N$, we set $\varphi_q(k) := L(s + kT,x,y)$. We have 
\begin{equation}\label{eq18}
   {\mathcal A}_1(L)(s,x_q,y_q) = \int_{\N} \varphi_q(k) d \chi(k).
\end{equation}
The B. Levi theorem provides
\begin{equation}\label{eq19}
\int_{\N} \liminf_{q \rightarrow + \infty}  \varphi_q(k) d \chi(k) \leq \liminf_{q \rightarrow + \infty} \int_{\N} \varphi_q(k) d \chi(k).
\end{equation}
Since $L(t,.,.)$ is l.s.c., we have
$$\liminf_{q \rightarrow + \infty}  \varphi_q(k) = \liminf_{q \rightarrow + \infty} L(s + kT,x_q,y_q) \geq L(s + kT, x,y),$$
therefore
$${\mathcal A}_1(L)(s,x,y) \leq \int_{\N} \liminf_{q \rightarrow + \infty}  \varphi_q(k) d \chi(k).$$
Using \eqref{eq18} and \eqref{eq19}, we obtain
$${\mathcal A}_1(L)(s,x,y) \leq  \liminf_{q \rightarrow + \infty} {\mathcal A}_1(L)(s,x_q,y_q),$$
and the conclusion.

($\beta$) Since ${\mathcal A}_1(-L)=-{\mathcal A}_1(L)$, it is a mere consequence of $(\alpha))$. ($\gamma$) immediately follows from ($\alpha$) and ($\beta$).

($\delta$) We set
\[C_1 := \{ s \in [0,T] : \forall k \in \N, \forall x \in \R^n, L(s + kT,x,.) \;{\rm is} \; {\rm convex} \}.
\]
Arguing as in the proof of ($\alpha$), we obtain that for $\mu$-a.e. $s \in [0,T]$, for all $k \in \N$ and for all $x \in \R^n$, $L(s+ kT,x,.)$ is convex.

Let $(s,x) \in C_1 \times \R^n$. Let $y$ ,$y_1 \in \R^n$ and $\lambda \in (0,1)$. Then we have 
\begin{multline*}
{\mathcal A}_1(L)(s,x, (1- \lambda) y + \lambda y_1)
=  (1- e^{-rT}) \sum_{k=0}^{+ \infty} e^{-rkT} L(t + kT,x, (1- \lambda) y + \lambda y_1)\\
\leq  (1- e^{-rT}) \sum_{k=0}^{+ \infty} e^{-rkT} ( (1- \lambda) L(t + kT,x,y) + \lambda  L(t + kT,x,y_1))\\
=   (1- \lambda)(1- e^{-rT}) \sum_{k=0}^{+ \infty} e^{-rkT}L(t + kT,x,y)
   +\lambda   (1- e^{-rT})   \sum_{k=0}^{+ \infty} e^{-rkT} L(t + kT,x,y_1)\\
= (1- \lambda) {\mathcal A}_1(L)(s,x,y) +   \lambda  {\mathcal A}_1(L)(s,x,y_1),
\end{multline*}
and the convexity is proven.

($\epsilon$) For any $(s,x,y) \in [0,T] \times \R^n \times \R^n$, we have
\begin{eqnarray*}
 {\mathcal A}_1(L)(s,x,y) &= &  (1- e^{-rT}) \sum_{k=0}^{+ \infty} e^{-rkT} L(t + kT,x,y)\\
\null & \geq &(1- e^{-rT}) \sum_{k=0}^{+ \infty} e^{-rkT} \rho(y) = \rho(y).
\end{eqnarray*}

($\zeta$) $[(s,x,y) \mapsto (s+kT,x,y) \mapsto L(s+ kT,x,y)]$ is measurable as a composition of measurable functions. Since a linear combination of measurable functions is measurable, and since a limit of measurable functions is measurable, ${\mathcal A}_1(L)$ is measurable.

($\eta$) From ($\gamma$), it is sufficient to prove that, for all $(x,y) \in  \R^n \times \R^n$, ${\mathcal A}_1(L)( \cdot, x,y)$ is measurable when $L( \cdot, x,y)$ is measurable.
Notice that $[s \mapsto L(s + kT, x,y)]$ is measurable as a composition of measurable functions, hence $[s \mapsto e^{-rkT} L(s + kT, x,y)]$ is measurable as a product of measurable functions. Finally, for any integer ${\ell}$, the map $[s \mapsto \sum_{k=0}^{\ell} e^{-rkT} L(s + kT, x,y)]$ is measurable as a finite sum of measurable functions, and ${\mathcal A}_1(L) (\cdot, x, y)$ is measurable as a limit of measurable functions.
\end{proof}

We can now state our first main result on existence of solutions for the problem
\begin{equation}\label{eq45}
\left.
\begin{array}{rl}
{\rm Minimize} & \int_0^{+ \infty} e^{-rt} L(t,x(t), x'(t)) dt\\
{\rm when} & x \in x_0 + W^{1,1}_{T,0}(\R_+,\mu_r; \R^n).
\end{array}
\right\}
\end{equation}
where $\eta \in \R^n$ is fixed and $x_0 \in  W^{1,1}(\R_+, \R^n) \cap \overline{P_T^0}^\alpha(\R_+, \R^n)$ satisfies $x_0(0) = \eta$.

\begin{theorem}\label{th25}  Let $L : \R_+ \times \R^n \times \R^n \rightarrow \R_+$ be a function which satisfies the following conditions.
\begin{itemize}
\item[(a)] $L$ is a Caratheodory function.
\item[(b)] For a.e. $t \in \R_+$, for all $x \in \R^n$, $L(t,x, \cdot)$ is convex on $\R^n$.
\item[(c)] There exists $\rho : \R^n \rightarrow \R_+$ such that $\lim_{\vert y \vert \rightarrow + \infty} \frac{\rho(y)}{\vert y \vert} = + \infty$ and, for all $(t,x,y) \in \R_+ \times \R^n \times \R^n$, $L(t,x,y) \geq \rho(y)$.
\item[(d)] There exists $\tilde{x}\in x_0 + W^{1,1}_{T,0}(\R_+,\mu_r; \R^n)$ such that
\[
\int_0^{+ \infty} e^{-rt} L(t,x(t), x'(t)) dt < +\infty.
\]
\end{itemize}
Then Problem (\ref{eq45}) possesses a solution.
\end{theorem}
\begin{proof}
 We consider the following problem
\begin{equation}\label{eq46}
\left.
\begin{array}{rl}
{\rm Minimize} & \int_0^T {\mathcal A}_1(L)(s, u(s), u'(s)) d s\\
{\rm when} & u \in W^{1,1}(0,T; \R^n), \; u(0) = u(T) = \eta.
\end{array}
\right\}
\end{equation}
From (a), using ($\eta$), ($\delta$) and ($\epsilon$) in Lemma \ref{lem41}, we get that ${\mathcal A}_1(L))$ has the following properties:
\begin{itemize}
\item ${\mathcal A}_1(L)$ is a Caratheodory function.
\item For a.e. $t \in [0,T]$, and for all $x \in \R^n$, the function ${\mathcal A}_1(L)(t,x,.)$ is convex.
\item For all $(s,x,y) \in [0,T] \times \R^n \times \R^n$, ${\mathcal A}_1(L)(s,x,y) \geq \rho(y)$, where $\rho$ is superlinear.
\end{itemize}
By (d), the problem \eqref{eq46} does not take $+\infty$ value and according to \cite{BGH} (Remark 1, p.115), it admits a solution $\hat{u}$.
Finally, since
\[
\int_{\R_+} e^{-rt} L(t,x(t), x'(t)) dt = \frac{1}{1- e^{-rT}} \int_0^T e^{-rs} {\mathcal A}_1(L)(s, x(s), x'(s)) ds,
\]
we obtain that $\hat{x} := {\mathcal E}_T(\hat{u})$ is a solution of  Problem \eqref{eq45}.
\end{proof}

This existence result can be extended to the Sobolev spaces $W^{1, \alpha}(\R_+, \R^n)$ with $\alpha \in (1, + \infty)$. Let us set
\begin{equation}\label{eq47}
\left.
\begin{array}{rl}
{\rm Minimize} & \int_0^{+ \infty} e^{-rt} L(t,x(t), x'(t)) dt\\
{\rm when} & x \in x_0 + W^{1,\alpha}_{T,0}(\R_+,\mu_r; \R^n).
\end{array}
\right\}
\end{equation}
where $\eta \in \R^n$ is fixed and $x_0 \in  W^{1,\alpha}(\R_+, \R^n) \cap \overline{P_T^0}^\alpha(\R_+, \R^n)$ satisfies $x_0(0) = \eta$.
\begin{theorem}\label{th44}
Let $L : \R_+ \times \R^n \times \R^n \rightarrow \R_+$ be a function which satisfies the following conditions.
\begin{itemize}
\item[(a)] $L$ is a Caratheodory function.
\item[(b)] For all $(t,x) \in \R_+ \times \R^n$, $L(t,x, \cdot)$ is convex.
\item[(c)] There exist $a \in L^1(\R_+, \mu_r; \R_+)$ and $b \in (0, + \infty)$ such that $L(t,x,y) \geq a(t) + b \cdot \vert y \vert^{\alpha}$ for a.e. $t \in \R_+$ and for all $(x, y) \in \R^n \times \R^n$.
\item[(d)] There exists $\tilde{x} \in  x_0 +W^{1, \alpha}_{T, 0}(\R_+, \mu_r ; \R^n)$ such that $I(\tilde{u}) < + \infty$.
\end{itemize}
Then Problem (\ref{eq31}) possesses a solution.
\end{theorem}
\begin{proof}
We consider the following problem
\begin{equation}\label{eq48}
\left.
\begin{array}{rl}
{\rm Minimize} & \int_0^T {\mathcal A}_1(L)(s, u(s), u'(s)) d s\\
{\rm when} & u \in W^{1,\alpha}(0,T; \R^n), \; u(0) = u(T) = \eta.
\end{array}
\right\}
\end{equation}
As in the proof of Theorem \ref{th25}, the assumptions (a), (b) and (d) imply the same properties for ${\mathcal A}_1$ on $[0,T]\times\R^n\times\R^n$. From assumption (c) we obtain that ${\mathcal A}_1(L)(s,x,y) \geq {\mathcal A}(a)(s) + b \cdot \vert y \vert^{\alpha}$ for a.e. $s \in [0,T]$ and for all $(x,y) \in  \R^n \times \R^n$, and using Theorem \ref{def-averaging} we know thet $ {\mathcal A}(a) \in L^1(0,T; \R)$.\\
Consequently all the assumptions of Theorem 4.1 in \cite{Da} (p. 82) are fulfilled for Problem (\ref{eq48}) which allows us to assert that there exists $\hat{u}$ a solution of Problem (\ref{eq48}). To conclude, it suffices to verify that $\hat{x} := {\mathcal E}_T(\hat{u})$ is a solution of Problem (\ref{eq47}).
\end{proof}

{\bf Remarks.}
\begin{itemize}
\item In theorem \ref{th44}, the assumption (d) is ensured as soon as $L$ is supposed to have a polynomial growth in the third variable. In this case, $L$ can even take nonpositive values, since upper integrals turn into regular integrals. Notice that such arguments can not extend to the $W^{1,1}$ setting.
\item According to \cite{BGH}, in both theorems \ref{th25} and \ref{th44}, the assumption (a) can be replaced by requiring that $L$ is globally measurable and that for a.e. $t\in\R_+$, the map $L(t,\cdot,\cdot)$ is lower semi-continuous.
\end{itemize}
\section{Necessary conditions of optimality}
We do not treat the question of the Euler-Lagrange equation in the setting $W^{1,1}$. Indded, for boundary value problems, the authors of \cite{BGH} say, in the point (b) in p. 139, that even the Euler-Lagrange equation may fail for solutions issued from a Tonelli's partial regularity theorem. We will here only consider the case $W^{1, \alpha}$ with $\alpha \in (1, + \infty)$. The Euler-Lagrange equation appears via a regularity result, under strictly stronger assumptions on hte Lagrangian. As in the previous section, we first prove some preliminary results about the properties of the averaged Lagrangian.

\smallskip

For any finite dimensional normed real vector space $E$ and for any map $\Phi : \R_+ \times \R^n \times \R^n \rightarrow E$, we consider the following properties :

\medskip

(P1) $\forall (x,y) \in \R^n \times \R_n$, $\forall \epsilon > 0$, $\exists \delta  > 0$, $\forall t, t_1 \in \R_+$, $\forall (x_1,y_1) \in \R^n \times \R_n$ s.t.
\[
( \vert t-t_1 \vert \leq \delta, \vert x-x_1 \vert \leq \delta, \vert y-y_1 \vert \leq \delta) \Longrightarrow \vert \Phi(t,x,y) - \Phi(t_1, x_1,y_1) \vert \leq \epsilon.
\]

\smallskip

(P2) $\Phi \in C^1(\R_+ \times \R^n \times \R_n, E)$ and $\forall (x,y) \in \R^n \times \R_n$, $\forall \epsilon > 0$, $\exists \eta > 0$ s.t.
$\forall t, t_1 \in \R_+$, $\forall (x_1,y_1) \in \R^n \times \R_n$,
\begin{multline*}
( \vert t-t_1 \vert \leq \eta, \vert x-x_1 \vert \leq \eta, \vert y-y_1 \vert \leq \eta) \Longrightarrow \\
\frac{\vert \Phi(t_1,x_1,y_1) - \Phi(t,x,y) - D \Phi(t,x,y)(t_1 -t, x_1 -x, y_1 -y) \vert}{(\vert t-t_1 \vert + \vert x-x_1 \vert +\vert y-y_1 \vert)} \leq \epsilon.
\end{multline*}

\smallskip

(P3) The partial differential $D_3\phi(t,x,y)$ exists for any $(t,x,y) \in  \R_+ \times \R^n \times \R_n$ and satisfies the following condition: $\forall (x,y) \in \R^n \times \R_n$, $\forall \epsilon > 0$, $\exists \beta = > 0$, $\forall z \in \R^n$,
\[
\vert z \vert \leq \beta \Longrightarrow
\forall t\in\R_+, \ \vert \Phi(t,x,y +z) - \Phi(t,x,y) - D_3 \Phi(t,x,y)z \vert \leq \epsilon \vert z \vert.
\]

\begin{lemma}\label{lem42}
let $L :  \R_+ \times \R^n \times \R_n \rightarrow \R$. The following assertions hold.
\begin{itemize}
\item[(i)] If $L$ satisfies (P1) then ${\mathcal A}_1(L) \in C^0([0,T] \times \R^n \times \R^n, \R)$.
\item[(ii)] If $L$ satisfies (P1) and (P3) and if $D_3L$ satisfies (P1) then $D_3 {\mathcal A}_1(L) \in  C^0([0,T] \times \R^n \times \R^n, {\mathfrak L}(\R^n,\R))$ and $D_3 {\mathcal A}_1(L) = {\mathcal A}_1( D_3 L)$.
\item[(iii)] We assume that $L \in C^1(\R_+ \times \R^n \times \R^n, \R)$, that $L$ satisfies (P1) and (P2) and that $DL$ satisfies (P1).\\
Then ${\mathcal A}_1(L) \in C^1([0,T] \times \R^n \times \R^n, \R)$ and we have $D {\mathcal A}_1(L) = A_1(DL)$.
\item[(iv)] We assume that $L \in C^2(\R_+ \times \R^n \times \R^n, \R)$, that $L$ and $DL$ satisfy (P1) and (P2), and that $D^2L$ satisfy (P1).\\
 Then ${\mathcal A}_1(L) \in C^2([0,T] \times \R^n \times \R^n, \R)$ and we have $D {\mathcal A}_1(L) = {\mathcal A}_1(DL)$ and $D^2 {\mathcal A}_1(L) = {\mathcal A}_1(D^2 L)$.
\end{itemize}
\end{lemma}
\begin{proof} We first prove (i). Let $(x,y) \in \R^n \times \R^n$ and $\epsilon > 0$. Let $s, s_1 \in [0,T]$, $(x_1,y_1) \in \R^n \times \R^n$ such that $\vert s-s_1 \vert \leq \delta$, $\vert x-x_1 \vert \leq \delta$ and $\vert y-y_1 \vert \leq \delta$. Then we have, for all $k \in \N$, $\vert (s + kT)- (s_1 + kT) \vert = \vert s-s_1 \vert \leq \delta$, and we get from (P1), for any $k \in \N$,
$$\vert L(s+kT,x,y) - L(s_1 + kT, x_1, y_1)\vert \leq \epsilon \Longrightarrow$$
\begin{multline*}
\vert {\mathcal A}_1(L)(s,x,y) - {\mathcal A}_1(L)(s_1,x_1,y_1) \vert \leq \\
(1-e^{-rT}) \sum_{k=0}^{+ \infty} e^{-rkT} \vert L(s+kT,x,y) - L(s_1 + kT, x_1, y_1)\vert \\ \leq 
(1-e^{-rT}) \sum_{k=0}^{+ \infty} e^{-rkT} \epsilon = \epsilon.
\end{multline*}

Let us now establish (ii). Let $(s,x,y) \in [0,T] \times \R^n \times \R^n$ and $\epsilon > 0$. Let $z \in \R^n$ such that $\vert z \vert \leq \eta$. From (P3), we have, for all $k \in \N$, 
$$\vert L(s+ kT,x,y + z) - L(s + kT,x,y) - D_3L(s+kT,x,y).z \vert \leq \epsilon \vert z \vert.$$
Then, setting $\rho := (1-e^{-rT})$, we obtain
\begin{multline*}
\null  \vert {\mathcal A}_1(L)(s,x,y+z) - {\mathcal A}_1(L)(s,x,y) - {\mathcal A}_1(D_3L)(s,x,y).z \vert \\
=  \rho\vert \sum_{k=0}^{+\infty} e^{-rkT} (L(s+kT,x,y+z) - L(s +kT,x,y) - D_3L(s+kT,x,y).z ) \vert \\
\leq  \rho \sum_{k=0}^{+\infty} e^{-rkT} \vert (L(s+kT,x,y+z) - L(s +kT,x,y) - D_3L(s+kT,x,y).z ) \vert\\
\leq  \rho \sum_{k=0}^{+\infty} e^{-rkT} \epsilon \vert z \vert = \epsilon \vert z \vert,
\end{multline*}

and the conclusion follows.

Next, we prove (iii). Let $(s,x,y) \in [0,T] \times \R^n \times \R^n$ and $\epsilon > 0$, and consider $\delta s$, $\delta x$ and $\delta y$ such that $\vert \delta s \vert \leq \eta$, 
$\vert \delta x \vert \leq \eta$ and $\vert \delta y \vert \leq \eta$, where $\eta = \eta(L, \epsilon, x,y)$ is provided by (P2). Then we have 
\begin{multline*}
\null  \vert {\mathcal A}_1(L)(s + \delta s, x + \delta x, y + \delta y) - {\mathcal A}_1(L)(s,x,y)  - {\mathcal A}_1(DL)(s,x,y) (\delta s, \delta x, \delta y) \vert\\
\leq  \sum_{k=0}^{+ \infty} e^{-rkT} \vert L(s + \delta s + kT, x + \delta x, y + \delta y) -  L(s  + kT, x , y ) \\
\null  - DL(s + kT,x,y) (\delta s, \delta x, \delta y) \vert\\
\leq \frac{1}{1 - e^{-rT}} \epsilon  ( \vert \delta s \vert + \vert \delta x \vert + \vert \delta y \vert).
\end{multline*}

These inequalities prove that ${\mathcal A}_1(L)$ is Fr\'echet differentiable at $(s,x,y)$ and that $D {\mathcal A}_1(L)(s,x,y) = {\mathcal A}_1 (DL) (s,x,y)$. Since $DL$ satisfies (P1), using (i), we can say that ${\mathcal A}_1(DL)$ is continuous, and consequently  $D({\mathcal A}_1(L))$ is continuous.

As a conclusion, (iv) is just a consequence of (i) applied to $L$ and to $DL$. 
\end{proof} 
\begin{theorem}\label{51}
Let $L : \R_+ \times \R^n \times \R^n \rightarrow \R$ be a function, where $n$ is a positive integer number, and let $\alpha \in (1, + \infty)$.\\
We assume that the following assumption are fulfilled.
\begin{itemize}
\item[(a)] $L$ is of class $C^2$ on $\R_+ \times \R^n \times \R^n $, $L$ and $DL$ satisfy the conditions (P1) and (P2) and $D^2L$ satisfies (P1).
\item[(b)] There exist constants $c_0, c_1 \in (0, + \infty)$ such that, 
$$\forall (t,x,y) \in \R_+ \times \R^n \times \R^n, c_0  \vert y \vert^{\alpha} \leq L(t,x,y) \leq c_1 (1 + \vert y \vert^{\alpha}).$$
\item[(c)] There exists a function $M : (0, + \infty) \rightarrow (0, + \infty)$ such that
\[
\begin{array}{l}
\forall (t,x,y) \in \R_+ \times \R^n \times \R^n, \vert x \vert^2 + \vert y \vert^2 \leq R^2 \Longrightarrow \\
\vert D_2 L(t,x,y) \vert + \vert D_3 L(t,x,y) \vert \leq M(R) (1 + \vert y \vert^2).
\end{array}
\]
\item[(d)] $\forall (t,x,y) \in \R_+ \times \R^n \times \R^n, \forall \xi \in \R^n \setminus \{ 0 \}, D_{33}L(t,x,y)(\xi, \xi) > 0$.
\end{itemize}
Suppose that $\hat{x}$ is a local solution of Problem (\ref{eq47}), then $\hat{x}$ is $C^2$ on $\R_+$ except at most at the points $kT$ where $k \in \N$, and satisfies the Euler equation
$$D_2L(t,\hat{x}(t), \hat{x}'(t)) = \frac{d}{dt} D_3L(t,\hat{x}(t), \hat{x}'(t))$$
for all $t \in \R_+ \setminus T\N $.
\end{theorem}
\begin{proof}
Form assumption (a) using Lemma \ref{lem42} (iv), we know that 
\begin{equation}\label{eq51}
A_1(L) \in C^2 \left(([0,T] \times \R^n \times \R^n, \R\right). 
\end{equation}
From assumption (b) we obtain
\begin{equation}\label{eq52}
\left.
\begin{array}{l}
\exists c_0, c_1 \in (0, + \infty), \forall (s,x,y) \in [0,T] \times \R^n \times \R^n,\\
 c_0  \vert y \vert^{\alpha} \leq A_1(L)(s,x,y) \leq c_1  (1 + \vert y \vert^{\alpha}).
\end{array}
\right\}
\end{equation}
Since $D_2 A_1(L) = A_1 (D_2 L)$ and $D_3 A_1(L) = A_1( D_3 L)$ after Lemma \ref{lem42}, from assumption (c) we obtain 
\begin{equation}\label{eq53}
\left.
\begin{array}{l}
\exists M : (0, + \infty) \rightarrow (0, + \infty), \forall (s,x,y) \in [0,T] \times \R^n \times \R^n,\\
 \vert x \vert^2 + \vert y \vert^2 \leq R^2 \Longrightarrow \\
\vert D_2 A_1(L)(s,x,y) \vert + \vert D_3 A_1(L)(s,x,y) \vert \leq M(R)  (1 + \vert y \vert^2).
\end{array}
\right\}
\end{equation}
Since $D_{33} A_1(L) = A_1(D_{33}L)$ after Lemma \ref{lem42},  from assumption (d) we obtain 
\begin{equation}\label{eq54}
\forall (s,x,y) \in [0,T] \times \R^n \times \R^n, \forall \xi \in \R^n \setminus \{ 0 \}, D_{33}A_1(L)(s,x,y)(\xi, \xi) > 0.
\end{equation}
With (\ref{eq51}), (\ref{eq52}), (\ref{eq53}), (\ref{eq54}), all the assumptions of the regularity theorem given in \cite{BGH}, p. 134, are fulfilled for Problem (\ref{eq48}) and so when $\hat{x}$ is a solution of Problem (\ref{eq47}), its restriction to $[0,T]$, $\hat{u}$, is a solution of Problem (\ref{eq48}). Then using the regularity theorem on Problem (\ref{eq48}), we obtain that $\hat{u} \in C^2([0,T], \R^n)$ and  $\hat{u}$ satisfies the Euler equation at each point of $[0,T]$. The conclusion is simply the translation on $\hat{x}$ of the properties of $\hat{u}$.
\end{proof}

\medskip

\noindent{\bf Aknowledgments.} We thank Professor Rabah Tahraoui (University of Rouen) for helpful discussions which permits us to improve the paper. We also warmly thank the referees for all the valuable comments that greatly improved the paper.

%%%%%%%%%%%%%%%%%%%%%%%%%%%%%%%%%%%%%%%%%%%%%%%%%%%%%
%%%%%%%%%%%%%%%%% Bibliographie %%%%%%%%%%%%%%%%%%%%%
%%%%%%%%%%%%%%%%%%%%%%%%%%%%%%%%%%%%%%%%%%%%%%%%%%%%%

%

\end{document}